\documentclass{amsart}

\usepackage{hyperref}
\usepackage{amsthm}
\usepackage{amsmath}
\usepackage{amssymb}
\usepackage{color}
\usepackage{graphicx}
\usepackage{url}
\usepackage[vlined, ruled, norelsize]{algorithm2e}
\usepackage{wrapfig}

\newtheorem{theorem}{Theorem}[section]
\newtheorem{proposition}[theorem]{Proposition}
\newtheorem{corollary}[theorem]{Corollary}
\newtheorem{lemma}[theorem]{Lemma}

\theoremstyle{definition}
\newtheorem{definition}[theorem]{Definition}
\newtheorem{example}[theorem]{Example}

\newcommand{\R}{\mathbb{R}}

\newcommand{\Trop}{\mathcal{T}}
\newcommand{\C}{\mathbb{C}}
\newcommand{\Q}{\mathbb{Q}}
\newcommand{\B}{\mathcal{B}}
\newcommand{\J}{\mathcal{J}}
\newcommand{\Lin}{\mathcal{L}}
\newcommand{\calQ}{\mathcal{Q}}
\newcommand{\N}{\mathcal{N}}

\DeclareMathOperator{\Image}{Im}
\DeclareMathOperator{\rowspace}{rowspace}
\DeclareMathOperator{\rank}{rank}

\begin{document}

\begin{abstract}
  We define and study the cyclic Bergman fan of a matroid $M$, which is a simplicial polyhedral fan supported on the tropical linear space $\Trop(M)$ of $M$ and is amenable to computational purposes. It slightly refines the nested set structure on $\Trop(M)$, and its rays are in bijection with flats of $M$ which are either cyclic flats or singletons. We give a fast algorithm for calculating it, making some computational applications of tropical geometry now viable. Our C++ implementation, called TropLi, and a tool for computing vertices of Newton polytopes of A-discriminants, are both available online. 
\end{abstract}

\title{\textsf{Computing Tropical Linear Spaces}}
\author{ \textsf{Felipe Rinc\'on} }
\thanks{ \textsf{ University of California, Berkeley, Berkeley, CA, USA. felipe@math.berkeley.edu .} }
\maketitle

\section{Introduction} 

Let $A$ be an $m \times n$ complex matrix of rank $m$, with columns $\textbf{a}_1$, $\textbf{a}_2, \dotsc,$ $\textbf{a}_n \in \C^m$. We denote by $M(A)$ its associated matroid, i.e., the matroid on the ground set $[n]:=\{1,2,\dotsc,n\}$ encoding the linear dependences in $\C^m$ among the columns of $A$. The circuits of $M(A)$ are then the sets $C \subseteq [n]$ such that there is a minimal dependence among the columns of $A$ of the form $\sum_{i \in C} \lambda_i \, \textbf{a}_i = 0$. See \cite{oxley} for an introductory reference to matroid theory.

The \textbf{tropical linear space} $\Trop(M)$ of any matroid $M$ over the ground set $[n]$ is the set of vectors $v \in \R^n$ such that for any circuit $C$ of $M$, the minimum $\min \{v_i \mid i \in C\}$ is attained at least twice (i.e., there exist $j,k \in C$ distinct such that $v_j = v_k = \min \{v_i \mid i \in C\}$). In the case where $M$ is the matroid associated to a complex matrix $A \in \C^{m \times n}$, $\Trop(M)$ agrees with the tropicalization of the linear subspace $\rowspace (A) \subseteq \C^n$ (using the trivial valuation on $\C$). We will not consider in this paper tropical linear spaces obtained by tropicalizing using non-trivial valuations; for a discussion about these more general tropical linear spaces and their beautiful combinatorics the reader is invited to see \cite{speyer}.

Tropical linear spaces are one of the most basic objects in tropical geometry, and interest in them has increased substantially in the last few years. They are the local building blocks for abstract smooth tropical varieties, they play a key role in defining a well-behaved tropical intersection product, and they are central objects for studying realizability questions in tropical geometry (see \cite{diagonal}, \cite{realizationspaces}, \cite{kristin}). They are also fundamental for the study of tropicalizations of varieties obtained as the image of a linear subspace under a monomial map \cite{tropdisc}.  It is thus desirable in many situations to have an explicit description of them as polyhedral fans, i.e., as a list of polyhedral cones in $\R^n$ on which it is possible to perform different computations.

There are several natural polyhedral fan structures that can be given to the tropical linear space of a matroid $M$. In \cite{feichtnersturmfels}, Feichtner and Sturmfels described a whole family $\N$ of polyhedral fans, all of them supported on the tropical linear space $\Trop(M)$. They compared these fans to the coarsest polyhedral structure on $\Trop(M)$, called the \textbf{Bergman fan} $\B(M)$ of $M$, which is induced by the normal fan of the matroid polytope associated to $M$. The finest fan structure in the family $\N$ is called the \textbf{fine subdivision} of $\Trop(M)$. It was studied by Ardila and Klivans in \cite{ardila}, where they used it to show that the intersection of the tropical linear space $\Trop(M)$ with the ($n-1$)-dimensional unit sphere is homeomorphic to the order complex of the lattice of flats of $M$, and thus to a wedge of spheres. The coarsest fan structure in the family $\N$ is called the (coarsest) \textbf{nested set fan} of $M$, and was studied in depth in \cite{feichtnersturmfels}. In particular, Feichtner and Sturmfels proposed an algorithm for computing the nested set fan of $M$ by gluing together ``local'' tropical linear spaces. In general, their algorithm has the inconveniences of having to go over all $\rank(M)!$ possible total orders on the elements of each basis of $M$, and of performing the computation of each maximal cone in the nested set fan a multiple number of times. 

In Section \ref{sectionfan} we introduce the \textbf{cyclic Bergman fan} $\Phi(M)$ of $M$, which is a simplicial polyhedral fan also supported on the tropical linear space $\Trop(M)$. The maximal cones of $\Phi(M)$ are described using some interesting combinatorial objects that we call ``compatible pairs''. We prove that the rays of $\Phi(M)$ are in correspondence with flats of the matroid $M$ that are either cyclic flats or singletons, showing that $\Phi(M)$ is in general a little finer than the nested set fan of $M$. In Section \ref{sectionalgorithm} we present an effective algorithm for computing the cyclic Bergman fan of any matroid $M$ that overcomes the difficulties present in \cite{feichtnersturmfels}. We carry out a C++ implementation of our algorithm in the case $M$ is the matroid associated to an integer matrix $A$. The resulting software, called \texttt{TropLi}, computes tropical linear spaces with great speed. It can also be used to compute basic matroidal information about the matrix $A$, like its collection of bases, circuits, or its Tutte polynomial. \texttt{TropLi} can be obtained at the website
\begin{center}
 \url{http://math.berkeley.edu/~felipe/tropli/} . 
\end{center}
In Section \ref{sectionimplementation} we give examples of a few computations done with it and report on its performance. Finally, in Section \ref{sectiondiscriminants} we describe how our computation of tropical linear spaces can be used to compute vertices of Newton polytopes of $A$-discriminants. A C++ implementation of this procedure is also available online.

\section{The Cyclic Bergman Fan}\label{sectionfan}

In this section we introduce the cyclic Bergman fan $\Phi(M)$ of a matroid $M$. It is a simplicial polyhedral fan supported on the tropical linear space $\Trop(M)$ of $M$ amenable to computational purposes. 

Let $M$ be any rank $m$ matroid on the ground set $[n]$ having no loops and no coloops. Suppose $I \subseteq [n]$ is an independent set of the matroid $M$ and $e \in [n]$ is an element not in $I$ such that $I \cup \{e\}$ is dependent. There is a unique circuit of $M$ contained in $I \cup \{e\}$ (containing the element $e$), which is called the \textbf{fundamental circuit} $C(e,I)$ of $e$ over $I$. It can be described as 
\begin{equation}\label{eqfundcirc}
  C(e,I) = \{e\} \cup \{ i \in I \mid I - \{i\} \cup \{e\} \text{ is independent} \}.
\end{equation}

Now, let $B \subseteq [n]$ be a basis of the matroid $M$. Let $\Sigma_B \subseteq \R^n$ be the polyhedral cone consisting of all vectors $v$ that make $B$ a basis of maximal $v$-weight, i.e., such that $\sum_{i \in B} v_i$ is maximal among all bases of $M$. The set $\Trop(M)_B := \Trop(M) \cap \Sigma_B$ is called the \textbf{local tropical linear space} of $M$ around the basis $B$.

\begin{example}\label{graphical}
Consider the $3 \times 6$ matrix
  \[ A := 
  \begin{pmatrix}
  1 & 1 & 0 & 0 & 1 & 0 \\
  0 & 0 & 1 & 1 & 0 & 1 \\
  0 & 0 & 0 & 0 & -1 & -1 \\
  \end{pmatrix}.
  \]
  The matroid $M := M(A)$ is in this case a graphical matroid, namely, the cycle matroid of the graph $G$ presented in Figure \ref{figgraph}. The circuits of $M$ correspond to minimal cycles of $G$ and the bases of $M$ correspond to spanning trees of $G$.  
  \begin{figure}[h]
    \centering
    \includegraphics[height=2cm]{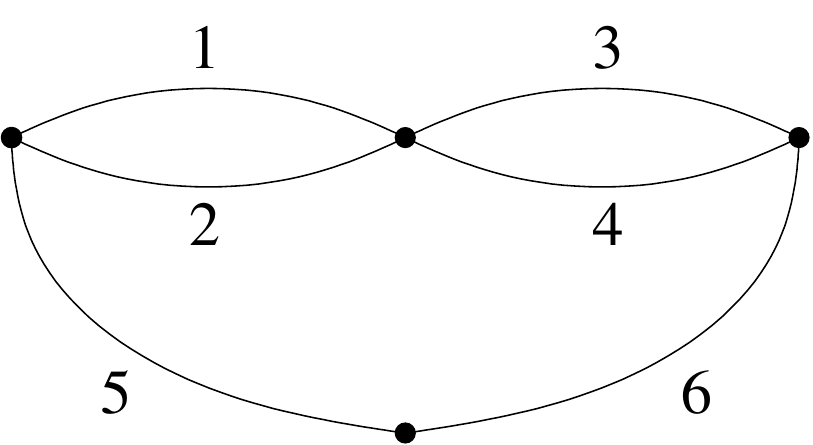}
    \caption{A graph $G$}
    \label{figgraph}
  \end{figure}
  
  \noindent The tropical linear space $\Trop(M)$ is then the set of vectors $v \in \R^6$ such that $v_1 = v_2$, $v_3= v_4$, and $\min(v_1,v_3,v_5,v_6)$ is attained twice. It is naturally a polyhedral fan with six maximal cones, corresponding to the six posibilities for the two positions where $\min(v_1,v_3,v_5,v_6)$ is attained.   For the basis $B := \{1,5,6\}$, the corresponding local tropical linear space $\Trop(M)_B$ is the set of vectors $v \in \Trop(M)$ that satisfy $v_3 = \min(v_1,v_3,v_5,v_6)$, which consists of only three of the six maximal cones described above. Note that each of these maximal cones is in several local tropical linear spaces. For example, the cone described by $v_5 \geq v_1 = v_2 = v_3 = v_4 \leq v_6$ is in the local tropical linear space corresponding to the bases $\{1,5,6\}, \{2,5,6\}, \{3,5,6\}$, and $\{4,5,6\}$. 
\end{example}

The following theorem and its corollary appear in the work of Feichtner and Sturmfels \cite{feichtnersturmfels}. They show that, although the tropical linear space $\Trop(M)$ might have a complicated combinatorial structure, all local tropical linear spaces are much faster to compute. In order to make our paper self contained, we give here a completely combinatorial proof of their result.
\begin{theorem}[\cite{feichtnersturmfels}]\label{onlyfund}
  Let $B$ be a basis of the matroid $M$. For any $v \in \Sigma_B$, $v$ is in the local tropical linear space $\Trop(M)_B$ if and only if the minimum $\min \{v_i \mid i \in C\}$ is attained at least twice for any \emph{fundamental} circuit $C$ over the basis $B$.
\end{theorem}
\begin{proof}
  Assume by contradiction that $v \in \Sigma_B$ is such that the minimum $\min \{v_i \mid i \in C\}$ is attained at least twice for all fundamental circuits $C$ over $B$, but $v$ is not in $\Trop(M)$. Let $D$ be a circuit of $M$ such that $\min \{v_i \mid i \in D\}$ is attained only once, and take $D$ containing as few elements outside of $B$ as possible. Since $D$ is not a fundamental circuit over $B$, the circuit $D$ contains at least two elements not in $B$. Let $a \in D$ be the element such that $v_a = \min \{v_i \mid i \in D\}$, and let $b \in D - B$ be different from $a$. Consider the fundamental circuit $C := C(b, B)$ of $b$ over $B$. Since $B$ is a basis of maximal $v$-weight, Equation \eqref{eqfundcirc} implies that the minimum $\min \{v_i \mid i \in C\}$ is attained at $b$, that is, $v_b \leq v_c$ for any $c \in C$. In particular, we have that $a \notin C$. Applying the strong circuit elimination axiom (see \cite[Proposition 1.4.11]{oxley}) to the circuits $C$ and $D$, with the elements $b \in C \cap D$ and $a \in D - C$, we get that there is a circuit $E \subseteq C \cup D - \{b\}$ containing the element $a$. But then $E$ is a circuit such that $\min \{v_i \mid i \in E\}$ is attained only once (at $i = a$), and $E$ has fewer elements outside of $B$ than the circuit $D$, which is a contradiction.
\end{proof}

\begin{corollary}\label{localbergman}
 Let $B = \{ b_1, b_2, \dotsc , b_m\} \subseteq [n]$ be a basis of the matroid $M$. The function $f_B:\R^m \to \R^n$ sending a vector $x = (x_1,x_2,\dotsc, x_m) \in \R^m$ to the vector $f_B(x) \in \R^n$ defined by
\[
  (f_B(x))_i := 
  \begin{cases}
    x_j & \text{if $i = b_j$ for some $j$,} \\
    \displaystyle{\min_{b_j \in C(i,B) - \{i\}} x_j} & \text{if $i \in [n] - B$;} 
  \end{cases}
\]
is a piecewise linear homeomorphism between $\R^m$ and the local tropical linear space $\Trop(M)_B$.
\end{corollary}
 
We now define the combinatorial objects that we will use to study the cyclic Bergman fan. Fix a basis $B \subseteq [n]$ of $M$. For any $k \in [n] - B$, denote $F_k := C(k,B) - \{k\}$ (note that $F_k \neq \emptyset$ since $M$ has no loops). Let $v$ be any vector in the local tropical linear space $\Trop(M)_B$, and suppose $\J$ is a total order on $B$ such that for any $a,b \in B$ we have $v_a < v_b \implies a <_{\J} b$ (note that for generic $v$ this condition determines $\J$ uniquely). This total order $\J$ induces a function $p:[n] - B \to B$ defined by $p(k) := \text{``$\J$-smallest element in $F_k$''}$. We say that $p$ is the \textbf{preference function} induced by the total order $\J$. According to Corollary \ref{localbergman}, this preference function $p$ is encoding which coordinates attain the minima described in Theorem \ref{onlyfund}, that is, $\min \{v_i \mid i \in C(k,B)\} = v_k = v_{p(k)}$ for all $k \in [n] - B$. Let $\Lin$ denote the restriction of the total order $\J$ to the image $\Image(p)$ of $p$. We call the pair $(p, \Lin)$ a \textbf{compatible pair} (with respect to the basis $B$) induced by the vector $v$. Note that a non-generic vector $v \in \Trop(M)_B$ might induce several different compatible pairs with respect to $B$, corresponding to different choices of the total order $\J$.

\begin{example}\label{preference}
 Let $A$ be the $4 \times 7$ matrix
 \[
   A :=
 \begin{pmatrix}
  1 & 0 & 0 & 0 & 0 & 3 & 1\\
  0 & 1 & 0 & 0 & 1 & 1 & 2\\
  0 & 0 & 1 & 0 & 0 & 0 & 1\\
  0 & 0 & 0 & 1 & 2 & 1 & 0
 \end{pmatrix}.
 \]
 Consider the basis $B := \{1,2,3,4\}$ of the matroid $M := M(A)$. The fundamental circuits over $B$ are $C(5,B) = \{2,4,5\}$, $C(6,B) = \{1,2,4,6\}$, and $C(7,B)= \{1,2,3,7\}$. Let $v = (0,5,2,3,3,0,0) \in \R^7$. It is not hard to see that the basis $B$ is a basis of maximal $v$-weight, so $v \in \Sigma_B$. Since the minimum $\min \{v_i \mid i \in C\}$ is attained at least twice for each fundamental circuit $C$ over the basis $B$, Theorem \ref{onlyfund} implies that $v \in \Trop(M)_B$. There is a unique total order $\J$ on the elements of $B$ satisfying $v_a < v_b \implies a <_{\J} b$, namely $1 <_\J 3 <_\J 4 <_\J 2$. The preference function $p$ induced by $\J$ is then given by $p(5)=4$, $p(6)=1$, and $p(7)=1$. The compatible pair $(p,\Lin)$ induced by $v$ in this way consists of the preference function $p$ together with the linear order $1 <_\Lin 4$. Note that a different order $\Lin$ on the image of $p$ would not be compatible with the preference function $p$. In fact, if the order $\J$ satisfied $4 <_\J 1$ then it would not be possible that $p(6)= 1$. 
\end{example}

\begin{proposition}\label{proptree}
  Let $(p,\Lin)$ be a compatible pair (with respect to the basis $B$). The set of vectors $v$ in the local tropical linear space $\Trop(M)_B$ that induce the pair $(p,\Lin)$ is an $m$-dimensional polyhedral cone $\Gamma(p,\Lin) \subseteq \R^n$. Its lineality space is generated by the vector $(1,1,\dotsc,1) \in \R^n$. After modding out by this lineality space, the cone $\Gamma(p,\Lin)$ is a simplicial polyhedral cone whose extremal rays can all be taken to be $0/1$ vectors. 
\end{proposition}
\begin{proof}
  Let $\calQ = \calQ(p)$ be the partition of the set $[n]$ with $m$ blocks $Q_b := \{b\} \cup p^{-1}(\{b\})$, for $b \in B$. Note that if a vector $v \in \Trop(M)_B$ induces the pair $(p,\Lin)$ then $v$ has to be constant on each of the blocks of $\calQ$, that is, for any $i,j$ in the same block of $\calQ$ we must have $v_i = v_j$. We will construct a directed caterpillar tree $T = T(p,\Lin)$ with set of vertices $\calQ$ encoding all further restrictions on the coordinates of such a vector $v$: If there is a directed path in $T$ from $Q_b$ to $Q_{b'}$ then $v$ must satisfy $v_i \leq v_j$ for $i \in Q_b$ and $j \in Q_{b'}$.

  Since $\Lin$ is a total order on the elements in the image of $p$, it naturally induces a total order on the non-singleton blocks of $\calQ$. We start the construction of $T$ as a directed path whose vertices are all the non-singleton blocks of $\calQ$, with their position in the path matching the order prescribed by $\Lin$ (i.e., it is possible to walk from $Q_b$ to $Q_{b'}$ if $b <_{\Lin} b')$. Now, for every $c \in B - \Image(p)$, add a directed edge from the non-singleton block $Q_b$ to the block $Q_c = \{c\}$, where $b$ is the $\Lin$-largest element in the image of $p$ for which there is a $k \in [n] - B$ such that $k \in Q_b$ (i.e. $p(k)=b$) and $c \in F_k$. Note that such a $b$ is guaranteed to exist since the matroid $M$ has no coloops. Corollary \ref{localbergman} ensures that the directed tree $T$ constructed in this way encodes precisely all the conditions on the coordinates of a vector $v$ for it to be a vector in the local tropical linear space $\Trop(M)_B$ inducing the pair $(p,\Lin)$. More specifically, $v \in \Trop(M)_B$ and $v$ induces $(p,\Lin)$ if and only if $v$ is constant on the blocks of $\calQ$ and for any directed edge $Q_b \to Q_{b'}$ in $T$ we have $v_b \leq v_{b'}$.

  Now, it is easy to see that the set $\Gamma(p,\Lin)$ of vectors $v$ satisfying the conditions imposed by $T$ is a polyhedral cone with lineality space generated by the vector $(1,1,\dotsc,1) \in \R^n$. Moreover, after modding out by its lineality space, the cone $\Gamma(p,\Lin)$ can be described as the positive span of $m-1$ linearly independent $0/1$ vectors, as follows. Think of $T$ as a partial order on the blocks of $\calQ$, and for any $b \in B$ define $w_b$ as the sum of all coordinate vectors $e_i$ such that $i$ is in the union of all blocks in $\calQ$ greater than or equal to $Q_b$ (according to $T$). The cone $\Gamma(p,\Lin)$ is then equal to the positive span of the vectors $\{w_b \mid b \in B \text{ and } w_b \neq (1,1,\dotsc,1)\}$.
\end{proof}

\begin{example}
  Let $M$ and $B$ be defined as in Example \ref{preference}. We showed that the pair $(p,\Lin)$ is a compatible pair, where $p$ is given by $p(5)=4$, $p(6)=1$, $p(7)=1$, and $\Lin$ is the total order $1 <_\Lin 4$. Following the proof of Proposition \ref{proptree}, the partition $\calQ(p)$ for this preference function is $\{\{1,6,7\},\{2\},\{3\},\{4,5\}\}$. The directed caterpillar tree $T$ associated to the pair $(p,\Lin)$ is depicted in Figure \ref{figtree}. It encodes the conditions for a vector $v \in \R^7$ for it to induce the compatible pair $(p,\Lin)$: $v$ induces $(p,\Lin)$ if and only if $v_3 \geq v_1 = v_6 = v_7 \leq v_4 = v_5 \leq v_2$. These equalities and inequalities define the simplicial polyhedral cone $\Gamma(p,\Lin)$ in $\R^7$. After modding out by the lineality space $\R \cdot (1,1,\dotsc,1)$, the cone $\Gamma(p,\Lin)$ is generated by the rays $e_2, e_{245}, e_3$.  
\begin{figure}[ht]
    \centering
    \includegraphics[height=3.4cm]{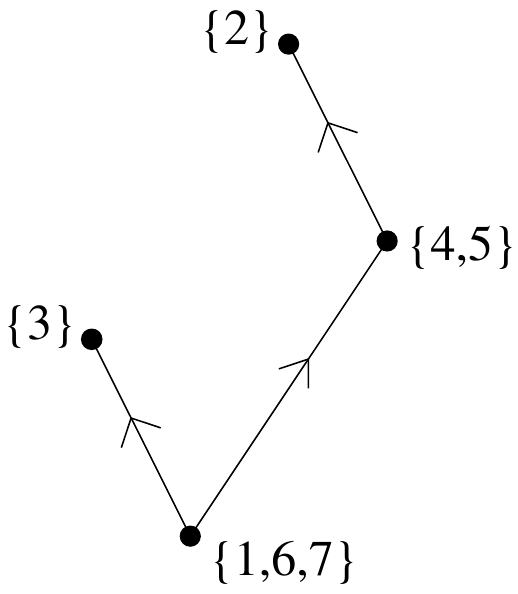}
    \caption{A directed caterpillar tree $T$}
    \label{figtree}
\end{figure}  
\end{example}

We will later prove in Theorem \ref{rays} that the extremal rays of the cones $\Gamma(p,\Lin)$ are precisely the indicator vectors of all flats of the matroid $M$ that are either cyclic flats or singletons.

It was pointed out to the author that the construction given in the proof of Proposition \ref{proptree} of the polyhedral cones $\Gamma(p,\Lin)$ from the directed trees $T(p,\Lin)$ agrees with a more general construction of Postnikov, Reiner, and Williams described in \cite{facespermutohedra}. In their paper, a ``braid'' polyhedral cone $\sigma_Q \subseteq \R^n$ is associated to every preposet $Q$ on the set $[n]$. In the case the Hasse diagram of the preposet $Q$ is a (directed) tree $T$, their construction of the cone $\sigma_Q$ agrees exactly with our construction of the cone $\Gamma(p, \Lin)$.   

It follows from our discussion that for any basis $B$ of $M$, the local tropical linear space $\Trop(M)_B$ is the union over all compatible pairs $(p,\Lin)$ with respect to $B$ of the simplicial cones $\Gamma(p,\Lin)$. However, since local tropical linear spaces corresponding to different bases might intersect nontrivially, compatible pairs with respect to different bases might give rise to the same cone. In order to find a canonical pair representing each cone, we say that a preference function $p:[n] - B \to B$ is \textbf{regressive} if $p(k) < k$ for all $k \in [n] - B$. If $(p,\Lin)$ is a compatible pair and $p$ is a regressive preference function, we say that $(p,\Lin)$ is a \textbf{regressive compatible pair}.

\begin{theorem}\label{thmfan}
  The tropical linear space $\Trop(M)$ is the union over all bases $B$ and all regressive compatible pairs $(p,\Lin)$ with respect to $B$ of the simplicial cones $\Gamma(p,\Lin)$. Moreover, if $(p_1,\Lin_1)$ and $(p_2,\Lin_2)$ are different regressive compatible pairs (possibly with respect to different bases $B_1$ and $B_2$), then the intersection of $\Gamma(p_1,\Lin_1)$ and $\Gamma(p_2,\Lin_2)$ is a proper common face of both cones.
\end{theorem}
\begin{proof}
  In order to show that $\Trop(M)$ is the union of all cones $\Gamma(p,\Lin)$ with $p$ regressive, let $v$ be any vector in $\Trop(M)$. Let $B$ be the first basis with respect to lexicographic order which has maximal $v$-weight (i.e., such that $\sum_{i \in B} v_i$ is maximal). The vector $v$ is then in the local tropical linear space $\Trop(M)_B$, so it induces some compatible pair $(p,\Lin)$ with respect to $B$. Assume by contradiction that $p$ is not a regressive preference function, so there exists a $k \in [n] - B$ such that $l := p(k) > k$. Since $l \in C(k, B)$, Equation \eqref{eqfundcirc} implies that $B' := B - \{l\} \cup \{k\}$ is also a basis of $M$. However, we have $v_k = v_l$, and thus $B'$ is also a basis of maximal $v$-weight, contradicting our choice of $B$.

  Now, suppose that $(p_1,\Lin_1)$ and $(p_2,\Lin_2)$ are any two compatible pairs. It follows from the description in terms of directed trees given in the proof of Proposition \ref{proptree} that the cones $\Gamma(p_1,\Lin_1)$ and $\Gamma(p_2,\Lin_2)$ intersect in a common face. Moreover, if $(p_1,\Lin_1)$ and $(p_2,\Lin_2)$ are distinct regressive compatible pairs then the corresponding directed trees $T(p_1,\Lin_1)$ and $T(p_2,\Lin_2)$ are different. It follows that this intersection has to be a proper face of both cones.
\end{proof}

\begin{definition}
  The \textbf{cyclic Bergman fan} $\Phi(M)$ of $M$ is the simplicial fan in $\R^n$ whose maximal cones are the cones $\Gamma(p,\Lin)$ with $(p,\Lin)$ a (regressive) compatible pair. The support of $\Phi(M)$ is the tropical linear space $\Trop(M)$ of the matroid $M$. 
\end{definition}

As we will see later, the cyclic Bergman fan structure $\Phi(M)$ on $\Trop(M)$ is a little finer than the (coarsest) nested set structure on $\Trop(M)$ that was described in \cite{feichtnersturmfels}. However, working with $\Phi(M)$ seems to be better for computational purposes, since its maximal cones are in one-to-one correspondence with effectively computable regressive compatible pairs.  An explicit example of how the different fan structures on $\Trop(M)$ might look like is given in Example \ref{diffans}.

We now study the rays of the cyclic Bergman fan $\Phi(M)$.

\begin{definition}
 A flat $F \subseteq [n]$ of $M$ is called a \textbf{cyclic flat} if it is equal to a union of circuits of $M$. Equivalently, $F$ is a cyclic flat if and only if $F$ is a flat of $M$ and $[n] - F$ is a flat of the dual matroid $M^*$.
\end{definition}

\begin{lemma}\label{lemmacyclic}
  Let $F$ be a cyclic flat, and suppose $I \subseteq F$ is an independent set spanning $F$. Then $F$ is a union of fundamental circuits over $I$.
\end{lemma}
\begin{proof}
  Denote by $U$ the union of all fundamental circuits over $I$ (which are contained in $F$), and assume by contradiction that $U \subsetneq F$. Since $F - I \subseteq U$, there exists some $i \in I$ such that $i \notin U$. Let $C \subseteq F$ be some circuit containing $i$ such that $|C - I|$ is as small as possible. Let $a$ be some element in $C - I$, and denote by $D$ the fundamental circuit of $a$ over $I$. Applying the strong circuit elimination axiom (see \cite[Proposition 1.4.11]{oxley}) to the circuits $C$ and $D$, with the elements $a \in C \cap D$ and $i \in C - D$, we get that there is a circuit $C' \subseteq F$ containing $i$ and contained in $C \cup D - \{a\}$, contradicting our choice of $C$.
\end{proof}

\begin{theorem}\label{rays}
  The rays of the cyclic Bergman fan $\Phi(M)$ (after modding out by the lineality space generated by the vector $(1,1,\dotsc,1) \in \R^n$) are precisely the rays generated by the vectors $e_F := \sum_{i\in F} e_i$, where $F \subsetneq [n]$ is a flat of $M$ which is either a cyclic flat or a singleton.
\end{theorem}
\begin{proof}
  Recall the description of the maximal cones of $\Phi(M)$ and their extremal rays in terms of directed trees given in the proof of Proposition \ref{proptree}. If $(p,\Lin)$ is a compatible pair  with respect to the basis $B$ then from this description we see that all the extremal rays of the cone $\Gamma(p,\Lin)$ have the form $\R_{\geq 0} \cdot e_F$, where $F = \{b\}$ for some $b \in B$ or $F \subsetneq [n]$ is a union of fundamental circuits over $B$. Moreover, since $e_F$ is in the tropical linear space $\Trop(M)$, $F$ must be a flat of $M$.
  
  Now, suppose $F = \{b\}$ is a flat of $M$. If $B$ is a basis of $M$ containing $b$ and $v$ is a generic vector in the local tropical linear space $\Trop(M)_B$ such that $v_b = \max_{a \in B} v_a$, then the singleton $\{b\}$ appears as one of the leaves in the directed tree corresponding to the compatible pair induced by $v$, so $e_b$ is an extremal ray of the corresponding maximal cone. 
  
  In the case $F$ is a cyclic flat, let $B$ be a basis of $M$ intersecting $F$ in as many elements as possible. The vector $v := e_F$ is then in the local tropical linear space $\Trop(M)_B$. Let $\J$ be any total order on $B$ satisfying $v_a < v_b \implies a <_{\J} b$, and let $(p,\Lin)$ be the compatible pair induced by $\J$. By Lemma \ref{lemmacyclic}, the directed tree associated to the pair $(p,\Lin)$ has a node $Q$ such that the set of elements that appear in nodes greater than or equal to $Q$ is precisely $F$. It follows that $e_F$ is an extremal ray of the cone $\Gamma(p,\Lin)$.
\end{proof}

We now discuss how the different fan structures on $\Trop(M)$ that have been studied in the literature compare to the cyclic Bergman fan $\Phi(M)$. Let us assume that the matroid $M$ is a connected matroid. The coarsest subdivision of $\Trop(M)$ is called the \textbf{Bergman fan} $\B(M)$ of $M$, and it was studied in \cite{feichtnersturmfels}. It is the fan structure on $\Trop(M)$ inherited from the normal fan to the matroid polytope of $M$. The rays in this fan (after modding out by the lineality space) are all the vectors of the form $e_F$ with $F$ a ``flacet'' of $M$ ($F \subseteq [n]$ is a ``flacet'' of $M$ if the matroid $M|F$ obtained by restricting to $F$ and the matroid $M/F$ obtained by contracting $F$ are both connected matroids). The Bergman fan is refined by the \textbf{nested set fan} (also studied in \cite{feichtnersturmfels}), whose rays are the vectors $e_F$ with $F$ a connected flat (i.e., a flat $F$ such that $M|F$ is connected), and whose maximal cones correspond to maximal nested sets of connected flats of $M$. This fan is in turn refined by the cyclic Bergman fan $\Phi(M)$, whose rays are the vectors $e_F$ with $F$ a flat which is either cyclic or a singleton, and whose maximal cones correspond to regressive compatible pairs. Finally, the cyclic Bergman fan $\Phi(M)$ is subdivided by the \textbf{fine subdivision} of $\Trop(M)$, which was studied in \cite{ardila}. In this fine subdivision the rays are the vectors $e_F$ with $F$ any flat, and the maximal cones correspond to maximal chains of flats. Since the last three of these fans are simplicial fans, one way of measuring how different these fan structures on $\Trop(M)$ are is to measure how different the following sets of flats of $M$ are:
\[\{F \text{ connected flat} \} \subseteq \{F \text{ flat, either cyclic or singleton} \} \subseteq \{F \text{ flat} \}.\]
A criterion for when the Bergman fan is equal to the nested set fan can be found in Theorem 5.3 of \cite{feichtnersturmfels}.

\begin{example}\label{diffans}
  Let $M$ be the graphical matroid defined in Example \ref{graphical}. The Bergman fan $\B(M)$ is the coarsest fan structure on the tropical linear space $\Trop(M)$, and it consists of the six maximal cones discussed in Example \ref{graphical}. This coarsest fan structure is also equal to the nested set fan of $M$. After modding out by the lineality space generated by the vector $(1,1,\dotsc,1) \in \R^6$, the fan $\B(M)$ has four rays $e_{12}, e_{34}, e_{5}, e_{6} \in \R^6$, corresponding to the four nontrivial connected flats of the matroid $M$ (which are also ``flacets'' of $M$). There is one more cyclic flat of $M$ which is not connected: the flat $\{1,2,3,4\}$. This implies that the cyclic Bergman fan $\Phi(M)$ strictly refines the fan $\B(M)$. In fact, the maximal cone of $\B(M)$ described by $v_1 = v_2 \geq v_5 = v_6 \leq v_3 = v_4$ gets subdivided into two smaller cones by the new ray $e_{1234}$ of the fan $\Phi(M)$. The fine subdivision of $\Trop(M)$ is a fan with ten rays, corresponding to the ten nontrivial flats of $M$. In the fine subdivision, each of the six maximal cones of $\B(M)$ gets subdivided by a new ray into two cones, to get a total of twelve maximal cones.
\end{example}

\section{Computing Compatible Pairs}\label{sectionalgorithm}

Let $M$ be a rank $m$ matroid on the ground set $[n]$ having no loops and no coloops. The cyclic Bergman fan $\Phi(M)$ described in Section \ref{sectionfan} allows us to develop an algorithm for computing the tropical linear space $\Trop(M)$ of $M$ in an effective way. As it was discussed above, the maximal cones of $\Phi(M)$ are in bijection with regressive compatible pairs, so the key idea for a fast calculation of $\Phi(M)$ lies in coming up with a good way of computing all possible regressive compatible pairs with respect to a given basis $B$. The way compatible pairs were defined made use of a total order $\mathcal{J}$ on the elements of $B$ to construct the pair, but it would not be a very good idea to go over all possible such total orders if $m$ is not very small. Instead, what we do is to construct recursively each compatible pair $(p,\Lin)$ by building up $p$ and $\Lin$ \emph{at the same time}. 
\begin{algorithm}\label{algobergman}
  \DontPrintSemicolon
  \LinesNumbered
  \SetKwFunction{pref}{Pref}
  \SetKwBlock{proc}{Procedure \pref{$k, p, \Lin$}:}{end}
  \SetKwFor{Each}{for each}{do}{end}

  \caption{Computing the cyclic Bergman fan $\Phi(M)$}
  \KwIn{A rank $m$ matroid $M$ having no loops and no coloops.}
  \KwOut{A list of all maximal cones in the cyclic Bergman fan $\Phi(M)$.}
  \BlankLine
  \Each{$B \subseteq [n]$ basis of $M$}{ 
    $\bullet$ Compute fundamental circuits:\;
    \Each{$k \in [n] - B$}{
      Compute $F_k := C(k,B) - \{k\}$.\;
    }
    $\bullet$ Compute recursively all regressive compatible pairs $(p,\Lin)$ with respect to $B$:\;
      Initialize $p = \emptyset$ and $\Lin = \emptyset$, and let $k$ be the first element in $[n] - B$.\;
      Apply the recursive procedure \pref{$k, p, \Lin$} described below.\;
      \proc{
        \uIf{$k = end$}{
          Output the constructed pair $(p,\Lin)$.\; 
        }
        \Else{  
          \If{$\Image(p) \cap F_k \neq \emptyset$}{ \nllabel{lineold}
            Define $p(k):=$ ``$\Lin$-smallest element in $\Image(p) \cap F_k$''.\;
            Let $k'$ by the first element in $[n] - B$ greater than $k$ (or $k'= end$ if $k$ is the last element in $[n] - B$).\;
            Apply \pref{$k',p,\Lin$}.\;
          }
          \Each{$b \in F_k - \Image(p)$ such that $b<k$}{ \nllabel{linenew} 
            \Each{total order $\Lin'$ on the set $\Image(p) \cup \{b\}$ that extends the total order $\Lin$}{ \nllabel{linenew2}
              \If{there is no $l < k$ in $[n] - B$ satisfying both $b \in F_l$ and $b <_{\Lin'} p(l)$}{  \nllabel{linecond}
                Define $p(k)=b$.\;
                Let $k'$ by the first element in $[n] - B$ greater than $k$ (or $k'= end$ if $k$ is the last element in $[n] - B$).\;
                Apply \pref{$k',p,\Lin'$}.\; 
              } 
            }
          }
        }
      }
    $\bullet$ Output the corresponding cones:\;
    \Each{pair $(p,\Lin)$ output in the previous step}{
      Compute the corresponding directed tree $T(p,\Lin)$, as described in the proof of Proposition \ref{proptree}.\;
      Output the cone $\Gamma(p,\Lin)$.\;
    }
  }
\end{algorithm}
Algorithm \ref{algobergman} describes a general procedure that achieves this goal. As we mentioned before, our algorithm has two important features that make it very fast compared to other existing algorithms: it does not have to go over all $m!$ total orders on the elements of each basis $B$, and moreover, each cone in the fan is computed exactly once, so there is no need to store them in memory or compare them with previously computed cones. 

The pseudocode for Algorithm \ref{algobergman} deserves some explanation. Its main part consists of the recursive procedure Pref$(k,p,\Lin)$, which computes for $k \in [n] - B$ all possible ways of defining $p(k)$ given that we have already computed $p(j)$ for all $j < k$ (with $j \in [n] - B$), and that we already have a total order $\Lin$ on $\Image(p) := \{p(j)\}_{j<k}$. The block starting on Line \ref{lineold} deals with the case where $p(k)$ is defined to be an element already in $\Image(p)$, in which case $p(k)$ can only be defined as the $\Lin$-smallest element in $\Image(p) \cap F_k$. The block starting on Line \ref{linenew} deals with the case where $p(k)$ is defined to be a new element not in $\Image(p)$. In this case, the condition on Line \ref{linecond} makes sure that the definition of $p(k)$ will not affect the compatibility of the pair $(p, \Lin)$.

\section{TropLi: A C++ Implementation}\label{sectionimplementation}

We developed a C++ implementation of the pseudocode described in Algorithm \ref{algobergman} for the case when the matroid $M$ is given as the matroid associated to an $m \times n$ integer matrix $A$ of rank $m$ (having no loops and no coloops). In this case, if $B$ is a basis of $M$ and $k \in [n] -B$, we compute the set $F_k := C(k,B) - \{k\}$ by first row-reducing the matrix $A$ in such a way that the submatrix of $A$ consisting of the columns indexed by $B$ is the identity, and then looking at the nonzero entries in the column indexed by $k$. A few minor changes were made to the pseudocode in Algorithm \ref{algobergman} in order to improve the efficiency of our implementation. For example, the order of the loops described by Line \ref{linenew} and Line \ref{linenew2} was reversed, so that the amount of times the condition in Line \ref{linecond} has to be checked is reduced significantly. 

In order to run through all bases $B$ of the matroid $M$, our code simply lists each subset of $[n]$ of size $m$ and tests directly if the corresponding columns are a basis of $\C^m$. It makes use of the C++ library LEDA \cite{leda} for carrying out all row operations on the matrix $A$ with exact integer arithmetic. A much more effective way of listing all bases of the matroid $M$ would be to make use of Avis and Fukuda's reverse search algorithm \cite{reversesearch}, which will be implemented in future versions of our code. 

The result is a fast software tool for computing the cyclic Bergman fan $\Phi(A):=\Phi(M(A))$ of an integer matrix $A$, called \texttt{TropLi}. This software, together with documentation on how to use it, are available online at the website 
\begin{center}
\url{http://math.berkeley.edu/~felipe/tropli/}. 
\end{center}
\texttt{TropLi} can also be used to compute some basic information about the matroid $M(A)$, like a list of all its bases, all its circuits, or its Tutte polynomial. 

We now present a few computations done using \texttt{TropLi} and report on its performance. All of the computations were performed on a laptop computer with a 2.0 GHz Intel Core 2 processor and 2 GB RAM.
\begin{example}
  Let $A$ be the $4 \times 8$ matrix whose columns correspond to the affine coordinates of the $8$ vertices of the three-dimensional unit cube. Running \texttt{TropLi} with this matrix $A$ as input takes just a few milliseconds, and produces lists of all rays and all maximal cones in the cyclic Bergman fan $\Phi(A)$. The first list shows that there are $20$ rays in the fan $\Phi(A)$, each of them specified as a $0/1$ vector in $\R^8$. The second list tells us that $\Phi(A)$ contains $80$ maximal cones, where each maximal cone is specified by its set of extremal rays. If instead we take $A$ to be the $5 \times 16$ matrix whose columns are the affine coordinates of the vertices of the four-dimensional unit cube, \texttt{TropLi} still takes a fraction of a second and computes $\Phi(A)$ to be a fan with $176$ rays and $2720$ maximal cones. 
  
  Running \texttt{TropLi} with the flag ``\texttt{-compare}'' produces a comparison between the cyclic Bergman fan $\Phi(A)$ and the Bergman fan $\B(A)$. For the three-dimensional cube these two fans are the same, and thus equal to the nested set fan. For the four-dimensional cube the Bergman fan has $2600$ maximal cones and it is thus a strict coarsening of the cyclic Bergman fan, even though the cyclic Bergman fan and the nested set fan are still equal (see Example 5.9 in \cite{feichtnersturmfels}). Our program outputs a list showing which maximal cones of $\Phi(A)$ are part of the same maximal cone in $\B(A)$. 
\end{example}

\begin{example}\label{exmixed1}
Consider the $4 \times 13$ matrix
\[
 A =
 \left(
 \begin{smallmatrix}
 1 & 1 & 1 & 0 &  0 &  0 &  0 &  0 &  0 &  0 &  0 &  0 &  0\\
 0 & 0 & 0 & 1 &  1 &  1 &  1 &  1 &  1 &  1 &  1 &  1 &  1\\
 0 & 1 & 0 & 0 & -1 &  0 & -2 & -1 &  0 & -3 & -2 & -1 &  0\\
 0 & 0 & 1 & 0 &  0 & -1 &  0 & -1 & -2 &  0 & -1 & -2 & -3\\
 \end{smallmatrix}
 \right). 
\]
 The orthogonal complement of the rowspace of $A$ is the rowspace of the $9 \times 13$ matrix
\[
  A^\perp =
 \left(
 \begin{smallmatrix}
 0 & 0 & 0 & 0 & -1 & 1 & 1 & -1 & 0 & 0 & 0 & 0 & 0\\
 0 & 0 & 0 & -1 & 1 & 0 & 0 & 1 & 0 & 0 & -1 & 0 & 0\\
 1 & 0 & -1 & 0 & 1 & -1 & 0 & 1 & 0 & 0 & -1 & 0 & 0\\
 0 & 0 & 0 & 2 & -2 & -1 & 0 & 0 & 0 & 0 & 1 & 0 & 0\\
 0 & 0 & 0 & 0 & 1 & -1 & 0 & 0 & 0 & -1 & 1 & 0 & 0\\
 0 & 0 & 0 & -1 & 1 & 0 & 0 & 0 & 1 & 0 & 0 & -1 & 0\\
 0 & -1 & 1 & 0 & -1 & 1 & 0 & -1 & 1 & 0 & 1 & -1 & 0\\
 0 & 0 & 0 & 0 & 0 & 0 & 0 & 0 & 0 & 1 & -2 & 1 & 0\\
 0 & 0 & 0 & 4 & -6 & 0 & 0 & 0 & 0 & 0 & 3 & 0 & -1\\
 \end{smallmatrix}
 \right). 
\]
 The matroid $M(A^\perp)$ is the dual matroid $M^*$ to the matroid $M:=M(A)$. As we will discuss in Section \ref{sectiondiscriminants}, computing the tropical linear space of this dual matroid can be used as the key ingredient for computing the tropicalization of the $A$-discriminantal variety. For this matrix $A$, this variety is the hypersurface defined by the condition on the coefficients of a general affine linear form $l(x,y)$ and a general cubic polynomial $g(x,y)$ so that the curves $l(x,y)=0$ and $g(x^{-1}, y^{-1})=0$ are tangent.

 A Maple implementation of the algorithm described in \cite{feichtnersturmfels} for computing tropical linear spaces locally takes many hours to compute $\Trop(M^*)$. As mentioned above, for each basis of the matroid $M^*$ (there are $430$ of them) it has to go through all $9! = 362 \, 880$ possible orderings of the rows of $A^\perp$. It also computes each maximal cone several times, so it has to compare each cone produced with the list of previously computed cones to see if it is a new cone or not. Running \texttt{TropLi} on the matrix $A^\perp$ computes the $9$ dimensional fan $\Phi(M^*)$ (which has $29$ rays and $2466$ maximal cones) in less than a second. 
\end{example}

\begin{example}\label{exmixed2}
  Let $A$ be the $5 \times 20$ matrix
  \[ 
    A = \left(
    \begin{smallmatrix}
    1 & 1 & 1 & 1 & 1 & 1 & 1 & 1 & 1 & 1 & 0 & 0 & 0 & 0 & 0 & 0 & 0 & 0 & 0 & 0 \\
    0 & 0 & 0 & 0 & 0 & 0 & 0 & 0 & 0 & 0 & 1 & 1 & 1 & 1 & 1 & 1 & 1 & 1 & 1 & 1 \\
    0 & 2 & 0 & 0 & 1 & 0 & 0 & 1 & 1 & 0 & 0 & 2 & 0 & 0 & 1 & 0 & 0 & 1 & 1 & 0 \\
    0 & 0 & 2 & 0 & 0 & 1 & 0 & 1 & 0 & 1 & 0 & 0 & 2 & 0 & 0 & 1 & 0 & 1 & 0 & 1 \\
    0 & 0 & 0 & 2 & 0 & 0 & 1 & 0 & 1 & 1 & 0 & 0 & 0 & 2 & 0 & 0 & 1 & 0 & 1 & 1 \\
    \end{smallmatrix} \right).
  \]
  The orthogonal complement of the rowspace of $A$ can be described as the rowspace of the $15 \times 20$ matrix
  \[ 
    A^\perp = \left(
    \begin{smallmatrix}
    -1 & 0 & 0 & 0 & 1 & 0 & 1 & 0 & -1 & 0 & 0 & 0 & 0 & 0 & 0 & 0 & 0 & 0 & 0 & 0 \\
    1 & -1 & 0 & 0 & 0 & 0 & 0 & 0 & 0 & 0 & -1 & 1 & 0 & 0 & 0 & 0 & 0 & 0 & 0 & 0 \\
    0 & 0 & 0 & 0 & 0 & 0 & 2 & 0 & -2 & 0 & -1 & 1 & 0 & 0 & 0 & 0 & 0 & 0 & 0 & 0 \\
    0 & 0 & 0 & 0 & 0 & 0 & 1 & 0 & -1 & 0 & -1 & 0 & 0 & 0 & 1 & 0 & 0 & 0 & 0 & 0 \\
    0 & -1 & 0 & 0 & 0 & 1 & 0 & 0 & 0 & 0 & 0 & 1 & 0 & 0 & 0 & -1 & 0 & 0 & 0 & 0 \\
    1 & -1 & 0 & 0 & 0 & 0 & -1 & 0 & 0 & 1 & 0 & 1 & 0 & 0 & 0 & -1 & 0 & 0 & 0 & 0 \\
    -1 & 0 & 0 & 0 & 0 & 0 & 5 & 0 & -4 & 0 & 0 & 2 & 1 & 0 & 0 & -2 & -1 & 0 & 0 & 0 \\
    1 & 0 & 0 & 0 & 0 & 0 & -1 & 0 & 0 & 0 & -1 & 0 & 0 & 0 & 0 & 0 & 1 & 0 & 0 & 0 \\
    2 & -1 & 0 & 0 & 0 & 1 & -5 & 0 & 3 & 0 & 0 & 0 & 0 & 1 & 0 & 0 & 0 & -1 & 0 & 0 \\
    0 & 0 & 0 & 1 & 0 & 0 & 0 & -1 & 0 & 0 & 0 & 0 & 0 & -1 & 0 & 0 & 0 & 1 & 0 & 0 \\
    1 & 0 & 0 & 0 & -1 & 0 & 0 & 0 & 0 & 0 & 0 & 0 & 0 & 0 & 0 & -1 & 0 & 1 & 0 & 0 \\
    0 & 0 & 0 & 0 & 0 & 0 & 0 & 0 & 0 & 0 & 0 & 0 & 0 & 0 & 0 & 1 & -1 & -1 & 1 & 0 \\
    0 & -1 & 0 & 0 & 0 & 0 & 0 & 1 & 0 & 0 & 0 & 1 & 0 & 0 & -1 & 0 & 1 & 0 & 0 & -1 \\
    0 & 0 & 0 & 0 & 0 & 0 & 1 & 0 & -1 & 0 & -1 & 0 & 0 & 0 & 0 & 0 & 1 & 1 & 0 & -1 \\
    0 & 0 & 1 & 0 & 0 & 0 & 3 & 0 & -4 & 0 & 0 & 2 & 0 & 0 & 0 & -3 & 0 & 0 & 0 & 1
    \end{smallmatrix} \right).
  \]
  Again, computing the tropicalization of the rowspace of $A^\perp$ can be used for studying the tropicalization of the $A$-discriminantal variety, which now corresponds to the condition on the coefficients of two general quadratic polynomials in three variables for their corresponding surfaces to be tangent.

  Running \texttt{TropLi} with the matrix $A^\perp$ as input computes the $15$-dimensional fan $\Phi(A^\perp) \subseteq \R^{20}$, which has $172$ rays and $475 \, 722$ maximal cones. All the computation takes just $60$ seconds. A lot of this time is actually spent testing which subsets of the set $\{1,2,\dotsc,20\}$ of size $15$ are bases of the matroid $M(A^\perp)$ by row-reducing all $15 \times 15$ submatrices of $A^\perp$. However, since $M(A^\perp)$ is the dual matroid to $M(A)$, we can instead use $A$ to compute the bases of $M$ (by reducing the $5 \times 5$ submatrices of $A$) and then take their complement to get the bases of $M(A^\perp)$. Also, the fundamental circuits $C_{M^*}(k,B)$ in $M^* = M(A^\perp)$ can be computed from the fundamental circuits in $M$, since $j \in C_{M^*}(k,B)$ if and only if $k \in C_{M}(j,[n]-B)$. This method of computing $\Phi(A^\perp)$ has been implemented in \texttt{TropLi}, and can be accessed by running it on the matrix $A$ using the flag ``\texttt{-dual}''. In this way, the computation of $\Phi(A^\perp)$ is even shorter: $30$ seconds.
\end{example}

\begin{example}\label{exmixed3}
  Let $A$ be the $6 \times 30$ matrix
  \[ 
    A = \left(
    \begin{smallmatrix}
    1 & 1 & 1 & 1 & 1 & 1 & 1 & 1 & 1 & 1 & 0 & 0 & 0 & 0 & 0 & 0 & 0 & 0 & 0 & 0 & 0 & 0 & 0 & 0 & 0 & 0 & 0 & 0 & 0 & 0 \\
    0 & 0 & 0 & 0 & 0 & 0 & 0 & 0 & 0 & 0 & 1 & 1 & 1 & 1 & 1 & 1 & 1 & 1 & 1 & 1 & 0 & 0 & 0 & 0 & 0 & 0 & 0 & 0 & 0 & 0 \\
    0 & 0 & 0 & 0 & 0 & 0 & 0 & 0 & 0 & 0 & 0 & 0 & 0 & 0 & 0 & 0 & 0 & 0 & 0 & 0 & 1 & 1 & 1 & 1 & 1 & 1 & 1 & 1 & 1 & 1 \\
    0 & 2 & 0 & 0 & 1 & 0 & 0 & 1 & 1 & 0 & 0 & 2 & 0 & 0 & 1 & 0 & 0 & 1 & 1 & 0 & 0 & 2 & 0 & 0 & 1 & 0 & 0 & 1 & 1 & 0 \\
    0 & 0 & 2 & 0 & 0 & 1 & 0 & 1 & 0 & 1 & 0 & 0 & 2 & 0 & 0 & 1 & 0 & 1 & 0 & 1 & 0 & 0 & 2 & 0 & 0 & 1 & 0 & 1 & 0 & 1 \\
    0 & 0 & 0 & 2 & 0 & 0 & 1 & 0 & 1 & 1 & 0 & 0 & 0 & 2 & 0 & 0 & 1 & 0 & 1 & 1 & 0 & 0 & 0 & 2 & 0 & 0 & 1 & 0 & 1 & 1 \\
    \end{smallmatrix} \right).
  \]
  The orthogonal complement of the rowspace of $A$ is then the rowspace of a $24 \times 30$ matrix $A^\perp$. The $A$-discriminant is in this case the condition on the Cayley octad obtained as the intersection of three general quadratic surfaces in 3-space (see \cite{quarticcurves}) to acquire a double point. Running $\texttt{TropLi}$ on the matrix $A^\perp$ we see that $\Phi(A^\perp)$ is a $24$-dimensional fan in $\R^{30}$ having $929$ rays and $154 \, 495 \, 683$ maximal cones. The computation of all these $150$ million cones takes a little more than $5$ hours. If instead we run $\texttt{TropLi}$ with the matrix $A$ as input and using the flag ``\texttt{-dual}'', the computation of $\Phi(A^\perp)$ takes less than $4$ hours.
\end{example}

\section{An Application: Computing A-Discriminants}\label{sectiondiscriminants}

Let $A$ be an $m \times n$ integer matrix of rank $m$ with columns ${\textbf a_1}, \dotsc, {\textbf a_n} \in \mathbb{Z}^m$, and suppose that the vector $(1,1, \dotsc,1)$ is in the rowspace of $A$. The columns of $A$ determine a collection of Laurent monomials ${\textbf x}^{\textbf a_1}, \dotsc, {\textbf x}^{\textbf a_n}$ in the ring $\C[x_1^{\pm 1},\dotsc,x_m^{\pm 1}]$ in a natural way. Consider the space $\C^A$ of all Laurent polynomials whose support is contained in this set of monomials, i.e., polynomials of the form $f(\textbf{x}) = \sum_{i=1}^n c_i \cdot {\textbf x}^{\textbf a_i}$, where the $c_i$s are complex coefficients. Let $\nabla_A \subseteq \C^A$ be the Zariski closure of the set of all $f$ in $\C^A$ that define a singular hypersurface in the torus $(\C^*)^m$, that is, for which there exists $\textbf{z} \in (\C^*)^m$ such that
\[
  f(\textbf{z}) = \frac{\partial f}{\partial x_i}(\textbf{z}) = 0 \,\, \text{ for all } i = 1,2,\dotsc, m.
\]
The variety $\nabla_A$ is an irreducible variety defined over $\Q$, called the \textbf{$A$-{dis\-cri\-mi\-nan\-tal} variety}. When $\nabla_A \subseteq \C^A$ is a subvariety of codimension $1$, the irreducible integral polynomial $\Delta_A$ in the coefficients of $f$ that defines $\nabla_A$ is called the \textbf{$A$-discriminant} ($\Delta_A$ is defined up to sign).

Of special interest is the case when $A$ is a general matrix whose first $s$ rows $r_1, \dotsc, r_s \in \mathbb{Z}^n$ are given by $r_j = \sum_{i \in I_j} e_i$ for some partition $\{I_1, \dotsc, I_s\}$ of $[n]$ (see examples \ref{exmixed1}, \ref{exmixed2}, and \ref{exmixed3}). In this case, the space $\C^A$ consists of polynomials of the form 
\[
f = x_1 \cdot f_1(x_{s+1}, \dotsc, x_m) + \dotsb + x_s \cdot f_s(x_{s+1}, \dotsc, x_m),
\]
where $f_j$ is a polynomial on the variables $\textbf{x}' = \{x_{s+1}, \dotsc, x_m\}$ whose support is contained in the set of monomials determined by the submatrix of $A$ with rows indexed by $\{s+1, \dotsc, m\}$ and columns indexed by $I_j$. The $A$-discriminantal variety is then the Zariski closure of the set of such polynomials $f_1(\textbf{x}'), \dotsc, f_s(\textbf{x}')$ that have a common root in the torus where their gradient vectors $(\partial f_j / \partial x)_{x \in \textbf{x}'}$ are linearly dependent. In the case where $s=2$ this corresponds to the condition on the polynomials $f_1, f_2$ for their corresponding hypersurfaces to be tangent. If $s = |\textbf{x}'|$ then this is the condition on the polynomials $f_1, \dotsc, f_s$ for the (finite) variety that they define to have a double point. In this case, the $A$-discriminant is also called their \textbf{mixed discriminant} (see \cite{mixed}).  If $s \geq |\textbf{x}'|+1$ and the matrix $A$ is \emph{essential} (see \cite{tropdisc}), then this is simply the condition on the polynomials $f_1, \dotsc, f_s$ for them to have a common root, so the $A$-discriminant is the same as their resultant. An extensive geometric treatment of all these notions can be found in \cite{GKZ}. 

Computing $A$-discriminants is in general a very hard computational task. Even for very small matrices $A$, the degree of $\Delta_A$ and its number of monomials can be quite large. From the definition, $A$-discriminants can in principle be computed by solving an elimination problem in the ring $\C[x_1^{\pm 1},\dotsc,x_m^{\pm 1}]$, but due to the huge size of these polynomials a Gr\"obner bases approach does not go too far. 

In \cite{tropdisc}, Dickenstein, Feichtner, and Sturmfels proposed a way of getting a handle on $A$-discriminants via tropical geometry. They proved that if $A^\perp$ denotes a Gale dual of the matrix $A$, i.e., an $(n-m) \times n$ matrix whose rowspace is equal to the orthogonal complement of the rowspace of $A$, then the tropicalization $\Trop(\nabla_A)$ of the variety $\nabla_A$ can be computed as the Minkowski sum of the tropical linear space $\Trop(M(A^\perp))$ and the rowspace of $A$. In the case where $\nabla_A$ has codimension $1$, this tropicalization $\Trop(\nabla_A)$ is equal to the $(n-1)$-dimensional skeleton of the normal fan of the Newton polytope $NP(\Delta_A)$ of $\Delta_A$. They used this to describe a ``ray shooting'' algorithm to recover vertices of $NP(\Delta_A)$ from $\Trop(\nabla_A)$, which goes as follows (see Theorem 1.2 in \cite{tropdisc}). Assume the columns of $A$ span the integer lattice $\mathbb{Z}^m$. Suppose $w$ is a generic vector in $\R^n$, and let $u \in \mathbb{Z}^n$ be the vertex of $NP(\Delta_A)$ minimizing the dot product $u \cdot w$. Then $u$ can be computed as
\begin{equation}\label{rayshoot}
 u_i = \sum_{\sigma \in \mathcal{C}_{i,w}} \left| \det (A^t, \sigma_1, \dotsc, \sigma_{n-m-1}, e_i) \right|,
\end{equation}
where $\mathcal{C}_{i,w}$ denotes the set of all maximal cones $\sigma$ in the nested set fan of $M(A^\perp)$ satisfying $(w + \R_{>0} \cdot e_i) \cap (\sigma + \rowspace A) \neq \emptyset$, and $\sigma_1, \dotsc, \sigma_{n-m-1}$ are the $0/1$ extremal rays of the cone $\sigma$ after modding out by its lineality space. An essential component in this procedure for computing vertices of $NP(\Delta_A)$ is to compute the tropical linear space $\Trop(M(A^\perp))$. It is possible, however, to replace in this formula the nested set fan of $M(A^\perp)$ by the cyclic Bergman fan $\Phi(A^\perp)$, which can be computed more easily.

Based on our implementation \texttt{TropLi} for computing cyclic Bergman fans, we developed a C++ code that computes vertices of Newton polytopes of $A$-discriminants in the way described above. Given an integer matrix $A$ and a vector $w \in \mathbb{Z}^n$, it computes a vertex $u$ of $NP(\Delta_A)$ minimizing the dot product $u \cdot w$. In the case where $w$ is not generic and this minimum is attained at several vertices of $NP(\Delta_A)$, the code uses a symbolic perturbation approach to compute one of these vertices at random. This software tool can also be obtained at the website 
\begin{center}
 \url{http://math.berkeley.edu/~felipe/tropli/} .
\end{center}

\begin{example}
  Consider the $4 \times 16$ matrix
  \[
  A = 
  \left(
  \begin{smallmatrix}
    1 & 1 & 1 & 1 & 1 & 1 & 0 & 0 & 0 & 0 & 0 & 0 & 0 & 0 & 0 & 0 \\
    0 & 0 & 0 & 0 & 0 & 0 & 1 & 1 & 1 & 1 & 1 & 1 & 1 & 1 & 1 & 1 \\
    0 & 0 & 1 & 0 & 1 & 2 & 0 & -1 & 0 & -2 & -1 & 0 & -3 & -2 & -1 & 0 \\
    0 & 1 & 0 & 2 & 1 & 0 & 0 & 0 & -1 & 0 & -1 & -2 & 0 & -1 & -2 & -3
  \end{smallmatrix}
  \right)
  \]
  Running our program with the matrix $A$ as input and the flag ``\texttt{-random 100}'' computes 100 random vertices of the Newton polytope of the $A$-discriminant $\Delta_A$. In this case, $\Delta_A$ is the condition on a general quadratic polynomial $f(x,y)$ and a general cubic polynomial $g(x,y)$ for the curves $f(x,y)=0$ and $g(x^{-1}, y^{-1})=0$ to be tangent. Our code computes the fan $\Phi(A^\perp)$ in the same way \texttt{TropLi} does (with the matrix $A$ as input and using the flag ``\texttt{-dual}''), and for each maximal cone $\sigma$ computed it checks if $\sigma + \rowspace(A) \subseteq \R^{16}$ has codimension $1$. If this is not the case then the cone $\sigma$ will not contribute to the sum in Equation \ref{rayshoot}. The cyclic Bergman fan $\Phi(A)$ has $18 \, 045$ maximal cones, $6 \, 675$ of which have codimension 1 after adding the rowspace of $A$. This initial computation takes $22$ seconds. The code then performs the ray shooting algorithm for 100 random values of the vector $w$ using the $6 \, 675$ cones computed before, and outputs the corresponding 100 vertices $u$ of the Newton polytope of $\Delta_A$. It also prints the $A$-degree of the $A$-discriminant $\Delta_A$, i.e., the vector $A \cdot u$ for $u$ any point of $NP(\Delta_A)$ (this does not depend on the choice of $u$). In this case, the $A$-degree is equal to $(24,22,-6,-6)$, so in particular we see that $\Delta_A$ is a homogeneous polynomial of degree $46$. This second part of the computation takes $7$ minutes.
\end{example}

Of course one would like to compute all the vertices of the Newton polytope $NP(\Delta_A)$. In general, however, due to the very large number of vertices of $NP(\Delta_A)$ and number of maximal cones in $\Trop(\nabla_A)$, one has to be quite clever about the way the vectors $w$ are chosen. Choosing them at random and waiting until all vertices of $NP(\Delta_A)$ have been computed is in general not viable. A great example illustrating all these difficulties and a few ways to overcome them is given in \cite{implicitization}. 

Very recently, an effective algorithm for recovering the normal fan of the Newton polytope of a polynomial from the support of its tropical hypersurface was proposed in \cite{tropicalresultants}. This algorithm has already been implemented in the software package Gfan \cite{gfan}. It can be used to take the description of $\Trop(\nabla_A)$ as a sum of a tropical linear space and a classical linear space and compute from it the normal fan of the Newton polytope of the $A$-discriminant $\Delta_A$. From this normal fan it is possible to recover the exact coordinates of the vertices of the Newton polytope of $\Delta_A$ by keeping track of the multiplicities of the codimension $1$ cones.

Now, even if we have a list of all the vertices of the Newton polytope $NP(\Delta_A)$, recovering the polynomial $\Delta_A$ is no easy task. One way to do it is to consider a generic polynomial whose monomials correspond to the lattice points in $NP(\Delta_A)$, and then imposing the condition that it vanishes on the image of the rational parametrization of $\nabla_A$ given in Proposition 4.1 of \cite{tropdisc}. This translates into a linear system of equations on the coefficients of this generic polynomial whose solution space corresponds to the coefficients of the $A$-discriminant $\Delta_A$. Note that, however, this procedure requires first computing all lattice points in the polytope $NP(\Delta_A)$ and then solving a very large system of linear equations, so a more effective approach would be desirable.

\section{Acknowledgements}
I am grateful to Bernd Sturmfels for conversations that led to this project and for useful comments on the presentation of this material. I would also like to thank the anonymous referees for carefully reading this manuscript and for their valuable suggestions.

\bibliographystyle{amsalpha}
\bibliography{../Bibliography/bibliography}

\end{document}